\documentclass[11pt, reqno]{amsart}

\usepackage{amsthm,amsmath,amssymb,booktabs,xcolor,graphicx}
\usepackage{bbm}
\usepackage{listings}
\usepackage{xcolor}

\definecolor{codegreen}{rgb}{0,0.6,0}
\definecolor{codegray}{rgb}{0.5,0.5,0.5}
\definecolor{codepurple}{rgb}{0.58,0,0.82}
\definecolor{backcolour}{rgb}{0.95,0.95,0.92}

\lstdefinestyle{mystyle}{
  backgroundcolor=\color{backcolour},   commentstyle=\color{codegreen},
  keywordstyle=\color{magenta},
  numberstyle=\tiny\color{codegray},
  stringstyle=\color{codepurple},
  basicstyle=\ttfamily\footnotesize,
  breakatwhitespace=false,         
  breaklines=true,                 
  captionpos=b,                    
  keepspaces=true,                 
  numbers=left,                    
  numbersep=5pt,                  
  showspaces=false,                
  showstringspaces=false,
  showtabs=false,                  
  tabsize=2
}

\usepackage[margin=1in]{geometry}

\usepackage[hidelinks]{hyperref}

\usepackage[utf8]{inputenc}
\usepackage[T1]{fontenc}

\usepackage[textsize=small,backgroundcolor=orange!20]{todonotes}

\usepackage[hidelinks]{hyperref}
\usepackage{url}

\usepackage[noabbrev,capitalize]{cleveref}

\usepackage[color,final]{showkeys} 

\colorlet{refkey}{orange!20}
\colorlet{labelkey}{blue!30}

\newtheorem{theorem}{Theorem}[section]
\newtheorem{proposition}[theorem]{Proposition}
\newtheorem{lemma}[theorem]{Lemma}
\newtheorem{claim}[theorem]{Claim}

\newtheorem*{question*}{Question}

\theoremstyle{definition}
\newtheorem{definition}[theorem]{Definition}

\newtheorem{question}[theorem]{Question}
\newtheorem*{definition*}{Definition}

\theoremstyle{remark}
\newtheorem*{remark}{Remark}

\makeatletter
\theoremstyle{theorem}
\newtheorem*{rep@theorem}{\rep@title}
\newcommand{\newreptheorem}[2]{%
\newenvironment{rep#1}[1]{%
 \def\rep@title{#2 \ref{##1}}%
 \begin{rep@theorem}}%
 {\end{rep@theorem}}}
\makeatother
\newreptheorem{theorem}{Theorem}
\newreptheorem{lemma}{Lemma}
\makeatother

\DeclareMathOperator{\Tr}{Tr}

\newcommand{\EE}{\mathbb{E}}

\newcommand{\RR}{\mathbb{R}}
\newcommand{\PP}{\mathbb{P}}

\newcommand{\NN}{\mathbb{N}}
\newcommand{\ZZ}{\mathbb{Z}}

\newcommand{\kAP}{\mathbf{kAP}}
\newcommand{\AP}{\mathbf{3AP}}
\newcommand{\nkAP}{\mathbf{\overline{kAP}}}

\newcommand{\Var}{\operatorname{Var}}
\newcommand{\Cov}{\operatorname{Cov}}

\newcommand{\tr}{\operatorname{Tr}}
\allowdisplaybreaks

\title{Number of arithmetic progressions in dense random subsets of $\ZZ/n\ZZ$}

\author[Berkowitz]{Ross Berkowitz}
\address{Department of Mathematics, Yale University, New Haven, CT 06520, USA}
\email{ross.berkowitz@yale.edu}

\author[Sah]{Ashwin Sah}
\address{Massachusetts Institute of Technology, Cambridge, MA 02139, USA}
\email{asah@mit.edu}

\author[Sawhney]{Mehtaab Sawhney}
\address{Massachusetts Institute of Technology, Cambridge, MA 02139, USA}
\email{msawhney@mit.edu}

\date{}

\begin{document}

\begin{abstract}
We examine the behavior of the number of $k$-term arithmetic progressions in a random subset of $\mathbb{Z}/n\mathbb{Z}$.  We prove that if a set is chosen by including each element of $\mathbb{Z}/n\mathbb{Z}$ independently with constant probability $p$, then the resulting distribution of $k$-term arithmetic progressions in that set, while obeying a central limit theorem, does not obey a local central limit theorem. The methods involve decomposing the random variable into homogeneous degree $d$ polynomials with respect to the Walsh/Fourier basis. Proving a suitable multivariate central limit theorem for each component of the expansion gives the desired result.
\end{abstract}

\maketitle

\section{Introduction}\label{sec:introduction}

Understanding the asymptotic behavior of sums of dependent random variables is a fundamental question in probability theory and combinatorics.  One particular random variable that has received attention is the number of arithmetic progressions in a random subset of $\mathbb{Z}/n\mathbb{Z}$.  For any subset $S\subseteq\mathbb{Z}/n\mathbb{Z}$ we define $\kAP(S)$ to count the number of $k$-term arithmetic progressions contained entirely in the set $S$. The probability space is constructed by choosing a random set $S$ by including each element of $\mathbb{Z}/n\mathbb{Z}$ independently at random with probability $p\in (0,1)$, where $p$ is a fixed constant not depending on $n$.  The natural question therefore is to understand the distribution of $\kAP(S)$ as $n$ grows.

It is not hard to show that $\kAP(S)$ obeys a central limit theorem.
That is, if we set $\mu_n=\EE[\kAP(S)]$ and $\sigma_n^2=\Var(\kAP(S))$, where $S$ is chosen as before, then for any fixed $a, b$
\[\PP\left[a\le \frac{\kAP(S)-\mu_n}{\sigma_n}\le b\right]=\frac{1}{\sqrt{2\pi}}\int_{a}^b\exp\left(-\frac{t^2}{2}\right)dx+o_{n,p}(1).\]

Given this Gaussian macroscopic behavior it is natural to guess that the distribution of $\kAP$ is ``smooth'' and therefore nearby integers are approximately as likely as one another. In particular, one may conjecture  that a local limit theorem estimating pointwise probabilities of $\kAP(S)$ that for any integer $x$

\[\PP[\kAP(S)=x]\stackrel{?}{=}\frac{1}{\sigma_n\sqrt{2\pi}}\exp\left(\frac{-(x-\mu_n)^2}{2\sigma_n^2}\right)+o\bigg(\frac{1}{\sigma_n}\bigg).\]

However the purpose of this note is to prove that this local limit theorem is in fact false and the distribution of $\kAP(S)$ oscillates wildly.

\begin{theorem}\label{thm:local-clt-failure}
Fix $p\in (0,1)$ and $k\ge 3$. Then for all sufficiently large $n$ relatively prime to $(k-1)!$ there is a point $x$ such that \[\Bigg|\mathbb{P}[\kAP(S) = x] - \frac{1}{\sigma_n\sqrt{2\pi}}\exp\bigg(\frac{-(x-\mu_n)^2}{2\sigma_n^2}\bigg)\Bigg| = \Omega\bigg(\frac{1}{\sigma_n}\bigg),\]
where $\mu_n$ and $\sigma_n$ in the statement are the expectation and standard deviation of $\kAP(S)$ and $S$ is constructed by choosing each element of $\ZZ/n\ZZ$ independently at random with probability $p$.
\end{theorem}
\begin{remark}
This failure of the local central limit likely extends to $\gcd(n,(k-1)!)\neq 1$, however the proof details become more technical and therefore we restrict our attention to this case.
\end{remark}

The first author discovered this failure of the local central limit theorem by sampling uniformly random subsets of $\mathbb{Z}/101\mathbb{Z}$ and counting the number of length 3 arithmetic progressions.  This histogram of results may be found in Figure \ref{experiment}.  Interestingly, it should be noted that subsequently and independently a study of Cai, Chen, Heller, and Tsegaye \cite{CCHT18} also conjectured that such a local limit theorem failed, but did not have a proof.

\begin{figure}
\begin{center}
\includegraphics[width=3in]{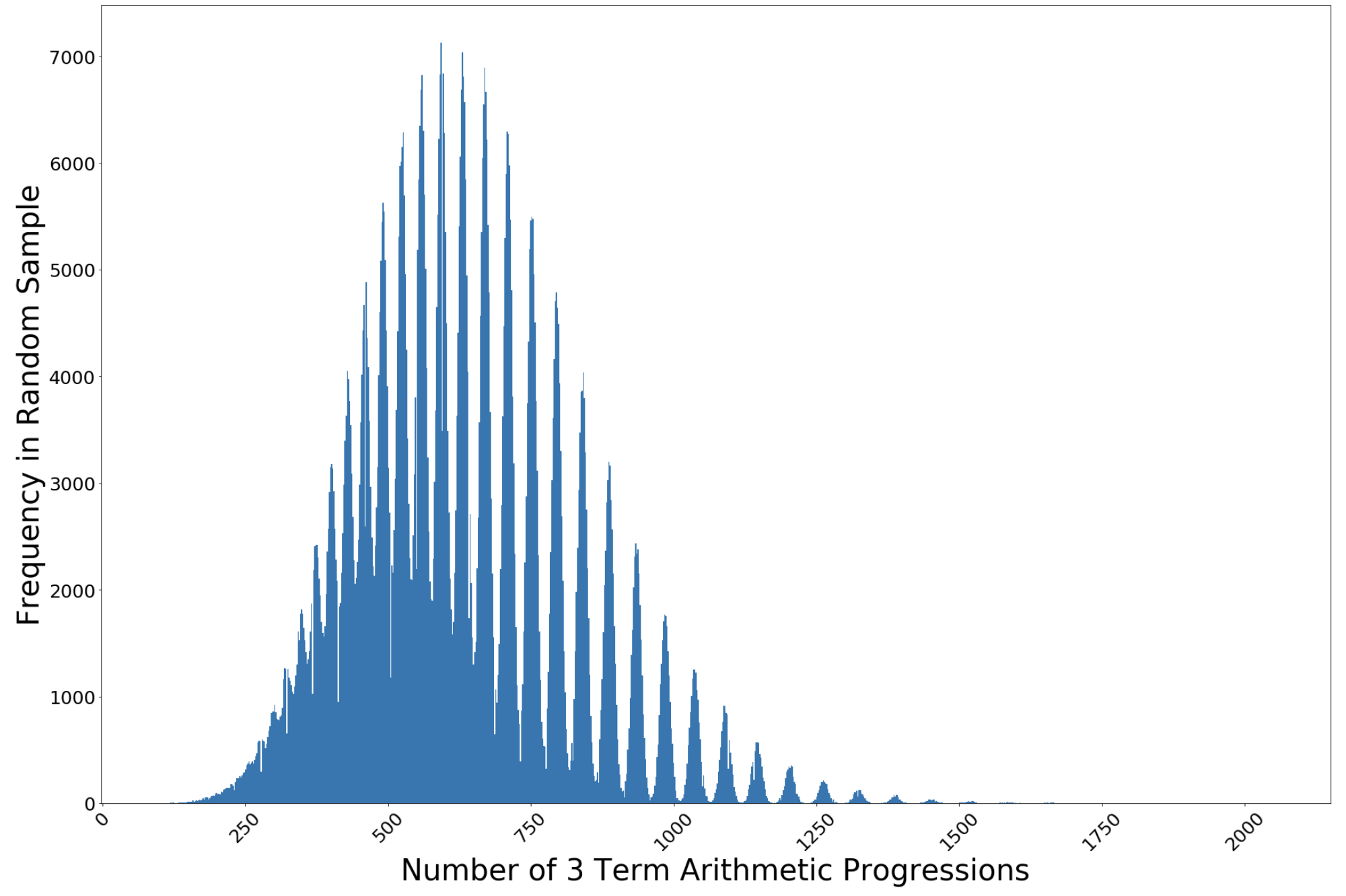} 
\includegraphics[width=3in]{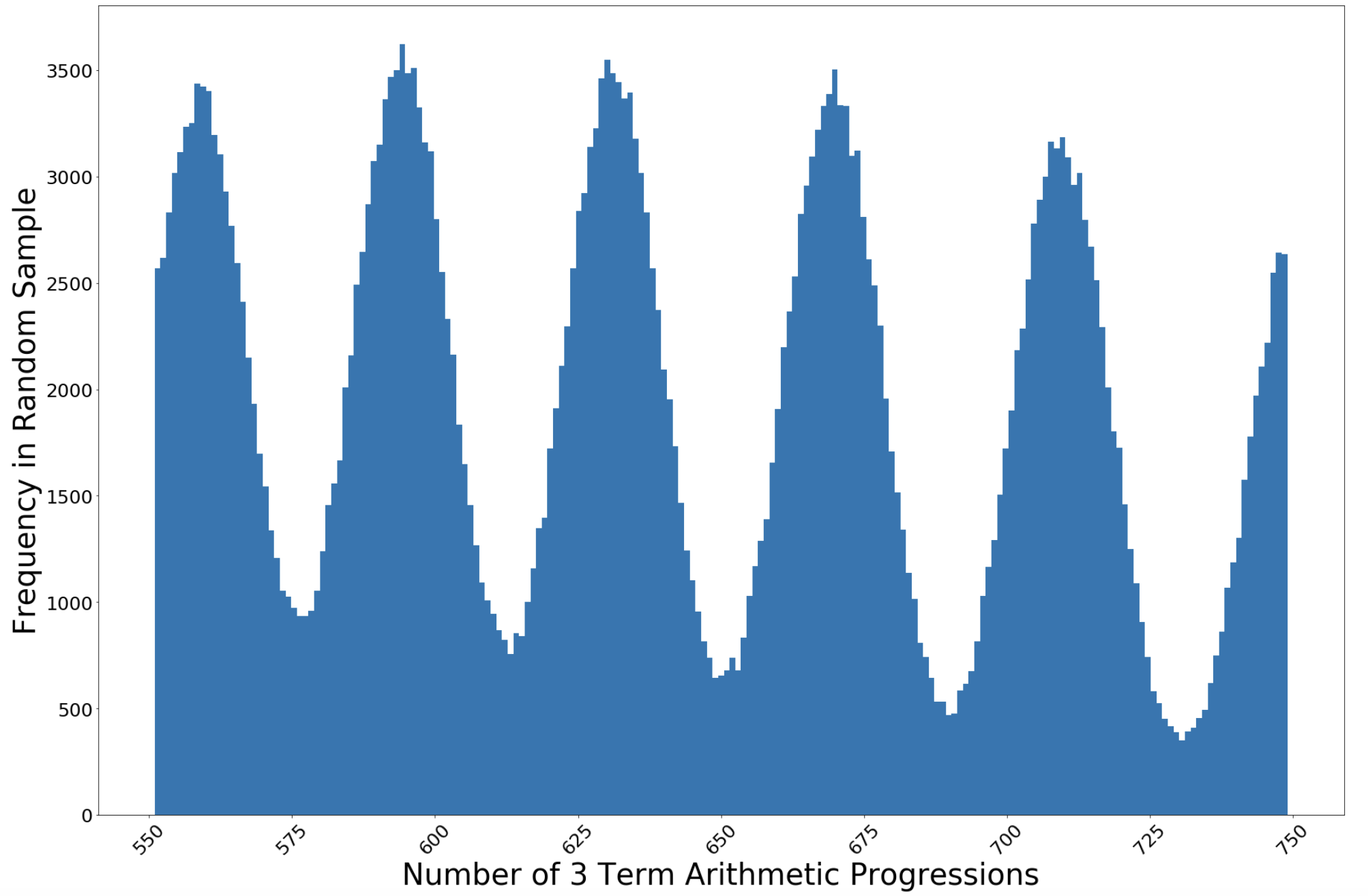}
\caption{
Histogram from sampling uniformly random subsets of $\mathbb{Z}/101\mathbb{Z}$ where 1,000,000 random samples were taken. While the the Gaussian-like distribution of $\AP(S)$ is visible there are wild local fluctuations.  The second picture on the right narrows the histogram to only looking at  $550\le \AP(S)\le 750$ showing the local fluctuations in greater detail.
}
\label{experiment}
\end{center}
\end{figure}

\subsection*{Related Work}

Significant attention has been given to understanding the large deviation probability of $\kAP(S)$, particularly in the sparse set regime where $p\to 0$. For example, recently Warnke \cite{W17}, Bhattacharya, Ganguly, Shao, and Zhao \cite{BGSZ20}, and Harel, Mousset, and Samotij \cite{HMS19} found precise upper tail bounds for $\kAP(S)$ in the sparse regime, while Janson and Warnke \cite{JW16} proved lower tail bounds. Additionally, Barhoumi-Andr\'eani, Koch, and Liu \cite{BKL19} proved a bivariate central limit theorem for $(\textbf{mAP}(S), \textbf{nAP}(S))$, understanding the joint distribution of the number of length $m$ and $n$ arithmetic progressions in sparse random sets.

Significant attention has also been focused on understanding local limit theorems in the analogous setting of $G(n,p)$. In particular work of Gilmer and Kopparty \cite{GK16} proves a local central theorem for triangle counts, with Berkowitz giving improved bounds in the case of triangles \cite{B1} and then proving the analogous theorem for cliques in \cite{B2}. Furthermore for general connected subgraph counts in $G(n,p)$ almost optimal anti-concentration results are known due to the work of Fox, Kwan, and Sauermann \cite{FKS19}. Our work points to a certain degree of separation between a local central limit theorem and anti-concentration for polynomial functions of Bernoulli random variables as already suggested by Fox, Kwan, and Sauermann \cite{FKS19}. In particular, the random variable $\kAP(S)$ experimentally appears to satisfy anti-concentration (at the optimal scale with each point probability being at most $O(1/n^{3/2})$) but as we will prove it does not satisfy a local central limit theorem.

\subsection*{Outline of Paper}
In \cref{sec:expansion} we compute the expansion of the $\kAP$ in the $p$-biased Fourier basis.  We then use this expansion to give a high level overview of our arguments. Sections \ref{sec:Reduction1}-\ref{sec:kolmogorov-distance} contain the main technical work of analyzing the asymptotic behavior of $\kAP$, and a more detailed overview of the argument can be found at the end of \cref{sec:Reduction1}.  The proof of the nonexistence of a local central limit theorem is in \cref{sec:local-clt-failure}. Finally, we end with some outstanding questions left by our work in \cref{sec:conclusion}.

\section{Expansion of $k$-AP function into $p$-biased Basis and Outline of the Argument} \label{sec:expansion}
We first expand the counting function of the number of $k$-APs into a $p$-biased Fourier basis. In order to do so define $x_i$ to be indicator if the element $i$ is in the subset of $\ZZ/n\ZZ$ which we are examining. Then we use the change of variables
\[y_i = \frac{x_i -p}{\sqrt{p(1-p)}}\]
and note that $\EE[y_i] = 0$ and $\Var[y_i] = 1$. Now the $k$-AP counting function is
\begin{align*}
\kAP(x) &= \sum_{a\in \ZZ/n\ZZ}\sum_{d\in [n/2]} \prod_{i=0}^{k-1}x_{a+id}\\
&=\sum_{a\in \ZZ/n\ZZ}\sum_{d\in [n/2]} \prod_{i=0}^{k-1}(y_{a+id}\sqrt{p(1-p)} + p)\\
&=\sum_{a\in \ZZ/n\ZZ}\sum_{d\in [n/2]} \sum_{\ell=0}^{k}\sum_{S\in {\binom{[k]}{\ell}}} p^{k - |S|}\prod_{i \in S} (y_{a+id}\sqrt{p(1-p)})\\
&=\sum_{\ell=0}^{k}\sum_{a\in \ZZ/n\ZZ}\sum_{d\in [n/2]} \sum_{S\in {\binom{[k]}{\ell}}} p^{k - \frac{|S|}{2}}(1-p)^{\frac{|S|}{2}}\prod_{i \in S} y_{a+id}.
\end{align*}
Furthermore define 
\[\kAP^{\ell}(y) = \sum_{a\in \ZZ/n\ZZ}\sum_{d\in [n/2]} \sum_{S\in {\binom{[k]}{\ell}}} p^{k - \frac{\ell}{2}}(1-p)^{\frac{\ell}{2}}\prod_{i \in S} y_{a+id}.\]
The key idea is to note that $\kAP^{\ell}$ for $\ell = 1$ versus all higher values of $\ell$ live on different scales. Our main lemma will be to prove a quantitative convergence of these components to $k$ appropriately scaled multivariate Gaussian and then using this analysis we will subsequently prove the desired failure of a local central limit theorem. For the sake of simplicity we also define 
\[\nkAP^{\ell}(y) = \frac{1}{\sigma_{\ell}} \sum_{a\in \ZZ/n\ZZ}\sum_{d\in [n/2]} \sum_{S\in {\binom{[k]}{\ell}}}\prod_{i \in S} y_{a+id}\]
where $\sigma_\ell$ is chosen so that $\Var[\nkAP^{\ell}(z)] = 1$, if $z$ is a vector of independent standard normals. In particular note $\sigma_\ell$ is independent of $p$. Note that these multivariate functions will be the central object of study and proving a sufficient strong result regarding their joint distribution will give the desired failure of a local central limit theorem. Finally define $\sigma$ to be the variance of $\kAP(z)$.

Note that all the different functions defined here are multilinear, since $\gcd(n,(k-1)!) = 1$.

\subsection*{Asymptotics of $\sigma_\ell$} We note that $\sigma_\ell = \Theta_k(n)$ is easily computed for $\ell \neq 1$ and that 
$\sigma_1 = \Theta_k(n^{\frac{3}{2}})$. This follows immediately from the fact that any two elements of $\ZZ/n\ZZ$ lie in $O_k(1)$ $k$-term arithmetic progressions jointly.

\subsection*{Overview of the Main Arguments} 
Given the above expansion we are now in position to give a general overview of the proof. The argument is centered on demonstrating that the functions $\kAP^{1}$ and the remaining $\kAP^{\ell}$ fluctuate independently and on differing scales and then use this to deduce a failure of the local central limit theorem. The key claim is that $\{\kAP^{\ell}\}_{\ell=1, 3\le \ell \le k}$, suitably normalized, approaches in distribution a set of independent Gaussian (along with quantitative bounds). In particular we prove the following result.
\begin{theorem}\label{thm:scaling-central-limit}
For $g\in C^3(\RR^{k-1})$, we have
\[|\EE g(\nkAP^{1}(y),\nkAP^{3}(y),\ldots,\nkAP^{k}(y)) - \EE g(Z)|\lesssim_{k}\frac{M_2(g)+M_3(g)}{n^{1/2}},\]
where $Z$ is a standard Gaussian random vector in $\RR^{k-1}$.
\end{theorem}
In the notation above $M_2(g)$ and $M_3(g)$ are the maximal operator norms of the second and third order derivative tensors of $g$; more informally these quantities simply measure the fluctuations in $g$. In order to prove the desired CLT result we first use a version of the Gaussian Invariance Principle which allows one to replace scaled Bernoulli's with Gaussians; this reduction appears in \cref{sec:Reduction1}. This reduction, while not strictly necessary, simplifies the argument. Then we use the theory of exchangeable pairs to deduce the necessary central limit theorem, which constitutes \cref{sec:exchangeable-pairs}. Roughly, the exchangeable pairs argument proceeds by using analyzing a single draw of Gaussian $y$'s for vector $(\nkAP^{1}(y),\nkAP^{3}(y),\ldots,\nkAP^{k}(y))$ and analyzes what occurs if precisely one of the $y$ at random is resampled. The method of exchangeable pairs allows one to deduce a quantitative central limit theorem from this perturbative analysis; however, a significant amount of effort is expended in verifying the necessary moment estimates. In particular, one consequence of the above analysis is that 
\[\frac{\sum_{a\in\ZZ/n\ZZ}\sum_{d\in [n/2]} z_{a}z_{a+d}z_{a+2d}}{\sqrt{\binom{n}{2}}}\overset{d}{\longrightarrow}\mathcal{N}(0,1)\]
if $z_i$ are independent standard normals, since this corresponds to $\nkAP^3$ for $k = 3$. Specializing the analysis to this case may be useful for some readers. Note here that if $z_{i}$ were not centered then the standard deviation jumps from $\Theta(n)$ to $\Theta(n^{3/2})$, and the corresponding central limit theorem is a easy consequence of the method of dependency graphs as demonstrated in \cite{CCHT18}.

We next convert these results into a bound on Kolmogorov distance and then deduce the failure of a local central limit theorem using a sampling argument in \cref{sec:kolmogorov-distance,sec:local-clt-failure}. Ultimately the failure of the local central theorem is derived essentially from the fact that $\kAP^{1}$ takes on values which are separated by $\Theta_{k,p}(n)$ from one another and the smearing which occurs due to the remaining components also lives on the scale of $\Theta_{k,p}(n)$. Thus it is not able to flatten this effect out. These two sections give one way of implementing this intuition.

\section{Reduction to Gaussian Estimate}\label{sec:Reduction1}
We first use an invariance principle for multilinear polynomials. In order to do so we need to define the influence of a variable for a Boolean function and whether a random variable is hypercontractive. This step is not strictly speaking necessary, but does not weaken our bounds and allows us to establish the rest of the argument in a slightly cleaner form.
\begin{definition}
The influence of a variable $x_i$ in a boolean function $F(x_1,\ldots,x_n) = \sum_{S\subseteq [n]} a_S\prod_{i\in S} x_i$ is 
\[\textbf{Inf}_t[F] = \sum_{t\in S\subseteq [n]} a_S^2.\]
\end{definition}
\begin{definition}
A random variable $X$ is $(p,q,\rho)$-hypercontractive ($1\le p\le q \le \infty$ and $0\le \rho < 1$) if for all constants $a,b\in \RR$ we have 
\[||a + \rho b X||_{q} \le ||a+bX||_{p}.\]
\end{definition}
Finally we need that the $p$-biased bit is hypercontractive with the appropriate constants. This follows from the following result in \cite{O14}.
\begin{theorem}
 If $X$ is a mean zero, symmetric, discrete random variable with
 \[\lambda = \min_{x\in \textit{Range}(X)} \PP[X = x].\] Then $X$ is $(2,3,\rho)$-hypercontractive for $\rho = \frac{1}{\sqrt{q-1}}\cdot \lambda^{\frac{1}{2}-\frac{1}{q}}.$
\end{theorem}
We are now ready to define a version of the multifunction invariance principle; this version appears in \cite{O14}. 
\begin{theorem}
Let $F^{(1)},\ldots, F^{(d)}$ be formal
n-variate multilinear polynomials each of degree at most $k\in \NN$. Let
$x_1,\ldots,x_n$ and $y_1,\ldots,y_n$ be independent $\mathbb{R}$-valued random variables
such that $\EE[x_t]=\EE[y_t]=0$ and $\EE[x_t^2]=\EE[y_t^2]=1$.
Assume each random variable $x_t$ and $y_t$ is $(2,3,\rho)$-
hypercontractive. Then for any $C^3$ function $\psi:\RR^d\to \RR$
satisfying $||\partial^\beta\psi||_{\infty}\le C$ for all $|\beta| = 3$,
\[\bigg|\EE[\psi(F(x))-\psi(F(y))]\bigg|
\le \frac{Cd^2}{3\rho^{3k}}\sum_{t=1}^{n}\sum_{j=1}^{d}\textbf{Inf}_t[F^{(j)}]^{3/2}\]
\end{theorem}
For our application this will amount to the following theorem on the distribution of $\nkAP(y)$ against test functions. 
\begin{theorem}\label{thm:invariance}
Let $y_i$ be defined as before and $y_i'$ be standard normal random variables. Then for any $C^3$ function $\psi:\RR^{k-1}\to \RR$
satisfying $||\partial^\beta\psi||_{\infty}\le C$ for all $|\beta| = 3$,
\[\bigg| \mathbb{E}[\psi(\nkAP^{1}(y'), \nkAP^{3}(y'),\ldots,\nkAP^{k}(y'))-\psi(\nkAP^{1}(y), \nkAP^{3}(y),\ldots,\nkAP^{k}(y))]\bigg|
\le \frac{C_{k,p}}{n^{1/2}}\]
where the constant $C_{k,p}$ is linearly proportional to $C$.
\end{theorem}
Here we used that $\textbf{Inf}_t[F^{(j)}] = O_{k,p}(1/n)$, which easily follows from the asymptotics of $\sigma_\ell$ along with the symmetry among the variables.

\section{Exchangeable Pairs}\label{sec:exchangeable-pairs}
We now consider the joint distribution of $\{\nkAP^{\ell}(y_i)\}_{1, 3\le \ell \le k}$ where the $y_i$ are now independent standard normals. We prove a quantitative result regarding its convergence to a multivariate normal using an application of exchangeable pairs. In order to state the version of exchange pairs, we will need we will first define a set of notations. Define for a real matrix 
\[\langle A, B\rangle = \Tr(A^TB)\]
and 
\[\|A\|_{HS} = \sqrt{\Tr(A^TA)} = \sqrt{\langle A, A\rangle}.\]
Furthermore define
\[\|A\|_{\text{op}} = \sup_{|v|=1,|w|=1}|\langle Av, w\rangle|\]
for matrices and similar for $k$-order forms 
\[\|A\|_{\text{op}} = \sup_{|v_i|=1}|A(v_1,v_2,\ldots,v_k)|.\] Given this define the $k^{\text{th}}$ derivative (tensor) operators for $f\in C^{k}(\mathbb{R}^n)$ as 
\[\langle D^kf(x),(u_1,\ldots,u_k)\rangle = \sum_{i_1,i_2,\ldots,i_k\in [n]}\frac{\partial^k f}{\partial x_{i_1}\ldots \partial x_{i_k}}(u_1)_{i_1}\ldots(u_k)_{i_k}\]
for vectors $u_1,\ldots,u_k\in\RR^n$. Finally define
\[M_r(g) = \sup_{x\in\mathbb{R}^n} \|D^rg(x)\|_{\text{op}}.\]
The last notion we need is that of exchangeable random variables. 
\begin{definition}
$X'$ and $X$ are exchangeable random variables if $(X',X)$ and $(X,X')$ have the same distribution.
\end{definition}
The key probability theoretic statement we will use is a multivariate version of exchangeable random variables for proving convergence to a Gaussian which in this form is due to Meckes \cite{M09}.
\begin{theorem}[\cite{M09}]\label{thm:exchangeable}
Let $(X,X')$ be an exchangeable pair of random vectors in $\RR^d$.  Suppose that
there is an invertible matrix $\Lambda$, and a 
random matrix $E'$ such that
\begin{itemize}
\item $\EE\left[X'-X\big|X\right]=-\Lambda X$
\item $\EE\left[(X'-X)(X'-X)^T\big|X\right]=2\Lambda+\EE\left[E'\big|X
\right].$
\end{itemize}
Then for $g\in C^3(\RR^d)$,
\begin{equation}
\big|\EE g(X)-\EE g(Z)\big|\le\|\Lambda^{-1}\|_{\text{op}}\left[
 \frac{\sqrt{d}}{4}M_2(g)
\EE\|E'\|_{HS}
+\frac{1}{9}M_3(g)\EE|X'-X|^3\right]
\end{equation}
where $Z$ is a standard Gaussian random vector in $\RR^d$.
\end{theorem}
Note this is a simplification of the statement which appears in \cite{M09} which is sufficient for our purposes. We now apply this to our setting where we set 
$y = (y_1,\ldots,y_n)$ and $y' = (y_1,\ldots, y_I',\ldots, y_n)$ where $I$ is a uniformly random coordinate in $[n]$ and $y_i, y_I'$ are independent standard normals. It obvious by definition that 
\[W = \{\nkAP^{\ell}(y)\}_{\ell=1, 3\le \ell \le k}, W' = \{\nkAP^{\ell}(y')\}_{\ell=1, 3\le \ell \le k}\]
are exchangeable random variables. We stress that the $\ell = 2$ term is missing. Furthermore, our normalization of $\sigma_\ell$ has made it so that each coordinate of $W$ has variance $1$. Also, $\EE[W] = 0$ since $\nkAP^{\ell}(y)$ is multi-linear as we have $\gcd(n,(k-1)!) = 1$. We first compute the matrix $\Lambda$ in the case of $(W,W')$.
\begin{proposition}\label{prop:exchange}
Let $W$, $W'$ be defined as above. Then 
\[\EE[W'-W|W] = -\mathbf{diag}\bigg(\frac{i}{n}\bigg)_{i=1, 3\le i\le k} W.\]
\end{proposition}
\begin{proof}
Note that the $\ell^{\text{th}}$ coordinate of $\EE[W'-W|y]$ is 
\begin{align*}
\frac{1}{n}\sum_{m=1}^{n}&\EE[\nkAP^{\ell}(y_1,\ldots,y_m',y_{m+1},\ldots)-\nkAP^{\ell}(y)|y]\\
&= \frac{1}{n}\sum_{m=1}^{n}\EE[\nkAP^{\ell}(y_1,\ldots,0,y_{m+1},\ldots)-\nkAP^{\ell}(y)|y]\\
&= -\frac{1}{n} \EE[\sum_{m=1}^{n}\frac{1}{\sigma_\ell} \sum_{a\in \ZZ/n\ZZ}\sum_{d\in [n/2]} \sum_{S\in {\binom{[k]}{\ell}}]}\prod_{i \in S, m\in a + dS} y_{a+id}|y]\\
&= -\frac{1}{n} \bigg(\EE[\ell(\nkAP^{\ell}(y))|y]\bigg) = -\frac{\ell}{n}\bigg(\nkAP^{\ell}(y)\bigg)
\end{align*}
and the proposition follows upon taking a conditional expectation with respect to $W$. The first equality follows because, conditional on the index $I = m$ which was removed, $y_m'$ is independent from everything else and has mean zero and $\nkAP^\ell$ is multilinear.
\end{proof}
Now to apply \cref{thm:exchangeable} we will simply take
\[E' = \EE[(W'-W)(W'-W)^T-2\Lambda|y],\]
which clearly satisfies the necessary hypothesis. To apply \cref{thm:exchangeable} to $W, W'$, we now see it suffices to bound $\EE[\|E'\|_{HS}]$ and $\EE|W'-W|^3$. Noting that $\|\Lambda^{-1}\|_{\text{op}} = n$, we will want bounds of the form $O_k(n^{-\frac{3}{2}})$ for each of these two quantities.

It is worth studying the diagonal terms more carefully. We have
\begin{align*}
\EE[(W_\ell'-W_\ell)^2] &= \EE[2W_\ell^2-2W_\ell W_\ell'] = \EE[2W_\ell\EE[W_\ell-W_\ell'|W_\ell]] = \frac{2\ell}{n}\EE[W_\ell^2] = \frac{2\ell}{n}
\end{align*}
by exchangeability, conditional expectations, \cref{prop:exchange}, and the normalization of $W_\ell$. Therefore
\begin{align*}
(E')_{\ell,\ell} &= \EE[(W_\ell'-W_\ell)^2|y] - \EE[(W_\ell'-W_\ell)^2]
\end{align*}
for $1\le\ell\le k$ and $\ell\neq 2$.

\subsection*{Computing $E'$}\label{subsec:computing-E-prime}
We begin by simply computing $E'$ entry by entry. For this we define the further refinement
\[\nkAP^{\ell,t}(y) = \frac{1}{y_t\sigma_\ell} \sum_{a\in \ZZ/n\ZZ}\sum_{d\in [n/2]} \sum_{S\in {\binom{[k]}{\ell}}}\prod_{i \in S, t\in a + d S} y_{a+id}.\]
Note that the above in theory is not defined when $y_t$ is $0$, but really we are simply removing the term $y_t$ from all products in the summation so this can be extended in the obvious way. Less formally this is the sum $\nkAP^\ell$ with all the terms containing the term $y_{t}$ with $y_{t}$ factored out. Using this notation it follows easily that the nondiagonal entries are  
\[(E')_{i,j} = \frac{1}{n}\sum_{s\in [n]} \nkAP^{i,s}(y) \nkAP^{j,s}(y)(y_s^2 + 1)\]
and the diagonal entries are 
\[(E')_{i,i} = -\frac{2i}{n} +  \frac{1}{n}\sum_{s\in [n]} (y_s^2 + 1)\nkAP^{i,s}(y)^2.\]
Here we used that each random variable $y_i$ has mean zero and variance one, and same for its replacement $y_I'$. In fact, the normalization of $\sigma_\ell$ implies that $\EE[(E')_{i,i}] = 0$, although this is not obvious by direct computation. We can see this from cross-comparison with the earlier expression for $(E')_{\ell,\ell}$.

We now proceed further into the computational abyss and consider 
\[\tr(E'E'^{T}) = \sum_{i,j} (E')_{i,j}^2.\]
We will need a bound on 
\[\EE\bigg[\sqrt{\tr(E'E'^{T})}\bigg]\le\sqrt{\EE[\tr(E'E'^{T})]}\]
which has decay properties of the form $O_k(n^{-3/2})$. Thus it suffices to prove a bound of the form $O_k(n^{-3})$ for
\[\EE[\sum_{i,j} (E')_{i,j}^2].\]
Note that it suffices to prove such a bound for each individual summands as there are $k^2$ such summands and the result will follow.

\section{Bounding $E'$}\label{sec:bounding-E-prime}
\subsection{Non-diagonal terms}\label{subsec:non-diagonal-terms}
We will first consider the case where $i\neq j$. Without loss of generality let $i < j$. Thus, since $i\ge 1$ and $j\neq 2$, we have $j\ge 3$.

We first define an even further refinement of our polynomials,
\[\nkAP^{\ell,t,S}(y) = \frac{1}{y_t\sigma_\ell} \sum_{a\in \ZZ/n\ZZ}\sum_{d\in [n/2]}\prod_{i \in S, t\in a + d S} y_{a+id}.\]
Note that
\[(E')_{i,j} = \frac{1}{n}\sum_{S\in \binom{[k]}{i}, T\in \binom{[k]}{j}}
\sum_{v\in [n]} \nkAP^{i,v,S}(y)\nkAP^{j,v,T}(y)(y_v^2+1)\]
and thus it suffices to prove that for all $S\in\binom{[k]}{i}$ and $T\in\binom{[k]}{j}$ we have
\[\EE\bigg[\bigg(\frac{1}{n} \sum_{v\in [n]} \nkAP^{i,v,S}(y)\nkAP^{j,v,T}(y)(y_v^2+1))\bigg)^2\bigg] = O_k(n^{-3})\]
by Cauchy-Schwarz. Using Cauchy-Schwarz again, it suffices to instead prove that
\[\EE\bigg[\bigg(\sum_{v\in [n]} \nkAP^{i,v,S}(y)\nkAP^{j,v,T}(y)\bigg)^2+
\bigg(\sum_{v\in [n]} \nkAP^{i,v,S}(y)\nkAP^{j,v,T}(y)y_v^2\bigg)^2\bigg] = O_k(n^{-1}).\]
The key idea is that when expanded as a polynomial in the $y$'s, the inside of the expectation will have most terms contain an odd power of some $y_i$, which leads to a zero contribution as $y_i$ is a standard normal. Every nonzero term contributes an amount bounded by $O_k(1)\cdot(\sigma_i\sigma_j)^{-2}$ as the exponents are bounded. Note that the parities of the exponents are unchanged between the first and second terms, so we use the second out of convenience.

A term of the second, when expanded out, amounts to choosing $v_1, v_2\in [n]$ and then $a_{11},a_{12},a_{21},a_{22}\in\ZZ/n\ZZ$ and $d_{11},d_{12},d_{21},d_{22}\in [n/2]$ such that $v_c\in a_{c1}+d_{c1}S$ and $v_c\in a_{c2}+d_{c2}T$ for $c\in\{1,2\}$. Let $A_{ct}$ for $k,t\in\{1,2\}$ be the sets thus formed (there are no self-intersections since $\gcd(n,(k-1)!) = 1$). Let the multiset $A$ equal the union, with repetition, of all $A_{ct}$. Note $|A_{c1}| = i$ and $|A_{c2}| = j$, so that $|A| = 2(i+j)$.

\begin{claim}\label{claim:small}
Given an initial choice of three of the $(a_{ct},d_{ct})$, there are $O_k(1)$ ways to choose the remaining pair so as to be compliant with the condition that every element in $A$ appears with even parity.
\end{claim}
\begin{proof}
After canceling we can see what parities the remaining set must have at each value of $\ZZ/n\ZZ$, which precisely determines it. (Recall that each is a set rather than a multiset because $\gcd(n,(k-1)!) = 1$.) Then there are $O_k(1)$ ways to choose the $(a_{ct},d_{ct})$ given what the set must be.
\end{proof}

First consider the case where $A_{11},A_{12}$ do not intersect at a place other than $v_1$. This means they together hit $i+j-2$ distinct values once, and $v_1$ twice. We see this immediately implies that $A_{21}$ and $A_{22}$, after removing $v_2$ from each, must hit precisely these $i+j-2$ distinct values. Since $j\ge 3$, this means that the value of $v_2$ is determined up to $O_k(1)$ choices by looking at the possible ways $A_{22}$ hits these $i+j-2$ values in $j-1\ge 2$ places. Therefore we see there are $O_k(1)$ choices of $a_{21},d_{21},a_{22},d_{22}$ after selecting $v_1,d_{11},d_{12}$, which means we have $O(n)^3\cdot O_k(1)$ total choices (as there are $O_k(1)$ choices of $a_{11},a_{12}$ given that information).

The analysis is similar if $A_{21},A_{22}$ do not intersect other than at $v_2$. So now we consider the case where $A_{11},A_{12}$ intersect at a place other than $v_1$ and same for $A_{21},A_{22}$. After choosing $v_1,v_2$ with $O(n^2)$ choices, we claim there are $O_k(1)$ ways to finish choosing. First note that $A_{11},A_{12}$ have $O_k(n)$ possibilities together, and after that there are $O_k(1)$ choices for whichever of $A_{21},A_{22}$ intersects $A_{11}\cup A_{12}$ (which must happen as these sets $A_{11}, A_{12}$ cannot cancel each other out, being of differing size). And by \cref{claim:small} there are $O_k(1)$ ways to choose the last one. This gives $O_k(n^3)$ once more.

In total, we have $O_k(n^3)$ terms that are not zero in the expectation. This yields a total contribution of $O_k(n^3)\cdot (\sigma_i\sigma_j)^{-2} = O_k(n^{-1})$, as desired.

\subsection{Diagonal terms}\label{subsec:diagonal-terms}
Now we consider the diagonal terms $(E_{\ell,\ell})'$. From \cref{sec:exchangeable-pairs}, we have
\begin{align*}
\EE[(E')_{\ell,\ell}^2] = \EE[(\EE[(W_\ell'-W_\ell)^2|y]-\EE[(W_\ell'-W_\ell)^2])^2] = \Var(\EE[(W_\ell'-W_\ell)^2|y]),
\end{align*}
where the variance in the final line is over the randomness of the standard normals $y$. Now
\begin{align*}
\EE[(W_\ell'-W_\ell)^2|y] = \EE_{I,y'}[(y_I'-y_I)^2\nkAP^{\ell,I}(y)^2] = \frac{1}{n}\sum_{i=1}^n(1+y_i^2)\nkAP^{\ell,i}(y)^2.
\end{align*}
If $\ell = 1$, $\nkAP^{\ell,i}(y)$ is a constant of size $\Theta_k(n^{-\frac{1}{2}})$. We end up with a bound of quality $O_k(n^{-3})$ trivially. Now let $\ell\ge 3$, recalling $\ell\neq 2$. We have
\begin{align*}
\EE[(E')_{\ell,\ell}^2] = \frac{1}{n^2\sigma_\ell^4}\sum_{1\le i,j\le n}\Cov\left[(1+y_i^2)(\sigma_\ell\nkAP^{\ell,i}(y))^2,(1+y_j^2)(\sigma_\ell\nkAP^{\ell,j}(y))^2\right].
\end{align*}

We first deal with the $i = j$ terms. They are bounded by $\EE[(1+y_i^2)^2(\sigma_\ell\nkAP^{\ell,i}(y))^4]$. Note that this value is independent of $i$, so we let $i = 0$ (which is the same as $i = n$). We adopt a similar method as before. It suffices to show that there are $O_k(n^2)$ terms of this $i = 0$ value that, when expanded, yield a nonzero expectation value. This is since there are $n$ such terms and as $\sigma_\ell = \Theta_k(n)$, which would lead to an overall contribution of the desired size $O_k(n^{-3})$ to $\EE[(E')_{\ell\ell}^2]$.

We use the $\nkAP^{\ell,i,S}$ refinement from before. By H\"older's inequality, it suffices to bound each $\EE[(1+y_0^2)^2(\sigma_\ell\nkAP^{\ell,0,S}(y))^4]$. We also can choose an exponent of $y_0$ from the initial term, but it does not affect parities of exponents and is of constant order so we ignore this and assume we have the $y_0^4$ term for simplicity.

Each term within this sum is chosen via $d_1,\ldots,d_4\in [n/2]$ and offsets (e.g. which term equals the $y_0$ term that is being divided in $\nkAP^{\ell,0,S}$), each of which contributes terms of the form $y_{d_jS_j}$, where $S_j$ is one of $O_k(1)$ many shifts of $S$ that contains $0$.  Now for any valid tuple $(d_1,\ldots,d_6)$, make a graph on vertex set $\{1,\ldots,6\}$, with $i,j$ connected if $d_iS_i$ and $d_jS_j$ intersect other than at $0$.

Given such a graph, we claim that there are $O_k(n)$ ways to choose the $d$s associated to a connected component of this graph in a manner compatible with the graph. Indeed, we find that, for example, if $d_1,d_2,d_3$ are connected, then choosing $d_1$ will fix the value say of $cd_2$ for some $c\in [k]$, which yields $O_k(1)$ possible values of $d_2$, and then $O_k(1)$ possible values of $d_3$ similarly.

Furthermore, in a graph associated to a valid tuple $(d_1,\ldots,d_6)$, i.e., one with even powers of the $y$s, we must have no disconnected vertices. Indeed, since the vertex is disconnected, the associated value of $d_i$ must give rise to a multiset $d_iS\pmod{n}$ which must cancel out all of its contributions to the $y$s. This is impossible as $\gcd(n,(k-1)!) = 1$.

Finally, we have at most $\frac{4}{2} = 2$ connected components of non-isolated vertices, each of which have $O_k(n)$ ways to choose the $d$s. This yields an upper bound of $O_k(n^2)$, as desired.

Now we consider $i\neq j$. Again, by translation invariance we see the value only depends on $j - i$. Therefore it suffices to show for $i\neq 0$ that
\[\Cov\left[(1+y_0^2)(\sigma_\ell\nkAP^{\ell,0}(y))^2,(1+y_i^2)(\sigma_\ell\nkAP^{\ell,i}(y))^2\right] = O_k(n)\]
since there are around $n^2$ total terms. We will show that
\[\Cov\left[(\sigma_\ell y_0\nkAP^{\ell,0}(y))^2,(\sigma_\ell y_i\nkAP^{\ell,i}(y))^2\right] = O_k(n).\]
Again, the other four cases that need to be verified will be essentially identical, since we are just removing even exponent terms. Consider
\begin{align*}
\EE[(\sigma_\ell y_0\nkAP^{\ell,0}(y))^2(\sigma_\ell y_i\nkAP^{\ell,i}(y))^2] &= \sum_{S,T\in\binom{[k]}{\ell}^2}\sum_{a,b\in(\ZZ/n\ZZ)^2}\sum_{\substack{d,e\in[n/2]^2\\0\in a_j+d_jS_j\\i\in b_j+e_jT_j}}\EE\left[\prod_{j=1}^2\prod_{s_j\in S_j} y_{a_j+d_js_j}\prod_{t_j\in T_j}y_{b_j+e_jT_j}\right].
\end{align*}

We claim that the amount of nonzero terms other than those with $a_1+d_1S_1=a_2+d_2S_2 = A$ and $b_1+e_1T_1=b_2+e_2T_2 = B$ and $A\cap B = \emptyset$ is $O_k(n)$. Indeed, this is a similar argument as in the non-diagonal term case. The analogue of \cref{claim:small} immediately follows, and similar arguments to earlier deal with the case that $a_1+d_1S_1,a_2+d_2S_2$ intersect only at $0$. Indeed, if so, then they must hit $2\ell-2$ distinct values once and $0$ twice. Then $b_1+e_1T_1,b_2+e_2T_2$, after removing $j$ from each, must precisely hit those $2\ell-2$ values, each hitting $\ell-1$ of them. As $\ell\ge 3$, we see that this pins down what $b_1,e_1,b_2,e_2$ are up to $O_k(1)$ possibilities, as linear combinations of the $d_i$ (say after fixing which positions of $a_i+d_iS_i$ equal $0$). Furthermore, we can compute $j$ as a linear combination of the $d_i$. Since $j\neq 0$, it must be a nonzero linear combination. This pins down $(d_1,d_2)$ to $O_k(n)$ possible values. Thus we have $O_k(n)\cdot O_k(1) = O_k(n)$ total possibilities. Furthermore, we have a similar analysis for the case where $b_1+e_1T_1,b_2+e_2T_2$ intersect only at $j$.

Now we consider the case where $a_1+d_1S_1,a_2+d_2S_2$ intersect other than at $0$, and similar for $b_i+e_iT_i$. Then there are $O_k(n)$ choices for $(a_1,d_1)$, and then $O_k(1)$ choices for $(a_2,d_2)$, and if we are assuming $a_1+d_1S_1\neq a_2+d_2S_2$, the terms $b_i+e_iT_i$ must hit one of these. There are thus $O_k(1)$ possibilities for the value of the set that does hit these values, and then $O_k(1)$ possibilities for the other set. Thus we have $O_k(n)$ choices once more, under the hypothesis that $a_1+d_1S_1\neq a_2+d_2S_2$. Similar analysis holds if instead we assume $b_1+e_1T_1\neq b_2+e_2T_2$. In the case where neither holds, so that $a_1+d_1S_1 = a_2+d_2S_2 = A$ and $b_1+e_1T_1 = b_2+e_2T_2 = B$, if we assume that $A\cap B\neq\emptyset$ then again we have a bound of $O_k(n)$,

The only remaining terms are those that have $a_1+d_1S_1 = a_2+d_2S_2 = A$ and $b_1+e_1T_1 = b_2+e_2T_2 = B$ and $A\cap B = \emptyset$, as desired.

Now we consider
\[\EE[(\sigma_\ell y_0\nkAP^{\ell,0}(y))^2] = \EE[(\sigma_\ell y_i\nkAP^{\ell,i}(y))^2] = \sum_{S\in\binom{[k]}{\ell}^2}\sum_{a\in(\ZZ/n\ZZ)^2}\sum_{\substack{d\in[n/2]^2\\0\in a_j+d_jS_j}}\EE\left[\prod_{j=1}^2\prod_{s_j\in S_j}y_{a_j+d_js_j}\right].\]
We claim that the only nonzero terms are those with $a_1+d_1S_1=a_2+d_2S_2$. Indeed, they are both sets as $\gcd(n,(k-1)!) = 1$, and if they are not equal then there is some element with an odd exponent.

Now, putting it all together, we see that
\[\Cov\left[(1+y_0^2)(\sigma_\ell\nkAP^{\ell,0}(y))^2,(1+y_i^2)(\sigma_\ell\nkAP^{\ell,i}(y))^2\right]\]
has a contribution of $O_k(n)$ terms in its $\EE[XY]$ portion which we bound by $O_k(n)$. Otherwise it only has terms corresponding to $a_1+d_1S_1=a_2+d_2S_2=A$ and $b_1+e_1T_1=b_2+e_2T_2=B$ and $A\cap B = \emptyset$. Thus the expectations factor into a product of $\EE[\prod_{a\in A}y_a^2]$ and $\EE[\prod_{b\in B}y_b^2]$. This is canceled by the terms described by
\[\EE[(\sigma_\ell y_0\nkAP^{\ell,0}(y))^2]\EE[(\sigma_\ell y_i\nkAP^{\ell,i}(y))^2].\]
There is one catch: the terms in this latter product of expectations $\EE[\prod_{a\in A}y_a^2]\EE[\prod_{b\in B}y_b^2]$ which have $A\cap B\neq\emptyset$ are not canceled in the $\EE[XY]$ term. However, we see that there are $O_k(n)\cdot O_k(1)$ of them, for after choosing whichever of $O_k(n)$ values for $A$ that we want, we have $O_k(1)$ choices of $B$ that both go through $j$ and intersect $A$ (unless $A$ goes through $j$, but then it only has $O_k(1)$ choices and $B$ has $O_k(n)$ choices in this case).

So, overall, the covariance is indeed $O_k(n)$, as desired.

\section{Bounding moments}\label{sec:bounding-moments}
For the second part we need to prove that $\mathbb{E}[|W-W'|^3]$ is of size $O_k(n^{-3/2})$ which is the same as
\[\EE[(\sum_{\ell=1, 3\le \ell\le k}|W^{(\ell)}-W'^{(\ell)}|^2)^\frac{3}{2}] = O_k(n^{-\frac{3}{2}}).\]
Therefore it in fact suffices to prove that 
\[\EE[|W^{(\ell)}-W'^{(\ell)}|^3]= O_k(n^{-\frac{3}{2}}),\]
using H\"older's inequality. For $\ell = 1$ this is trivial, using $\sigma_1 = \Theta_k(n^{3/2})$. Now let $\ell > 1$ and note that 
\begin{align*}
\EE[(|W^{(\ell)}-W'^{(\ell)}|)^3] &= \EE[|\nkAP^{\ell}(y)-\nkAP^{\ell}(y')|^3]\\
&= \EE_{t\in\ZZ/n\ZZ}\EE[|y_t-y_t'|^3|\nkAP^{\ell,t}(y)|^3]\\
&= \EE[|y_0-y_0'|^3|\nkAP^{\ell,0}(y)|^3]\\
&\le \EE[|y_0-y_0'|^6]^{\frac{1}{2}}\EE[|\nkAP^{\ell,0}(y)|^6]^{\frac{1}{2}}\\
&= \sqrt{120}\EE[\nkAP^{\ell,0}(y)^6]^{\frac{1}{2}}.
\end{align*}
We now use the $\nkAP^{\ell,t,S}(y)$ refinement as in \cref{subsec:non-diagonal-terms}. Using H\"older's inequality again, it suffices to prove that
\[\EE[\nkAP^{\ell,0,S}(y)^6] = O_k(n^{-3}).\]
We adopt a similar method to \cref{sec:bounding-E-prime}. Since $\sigma_\ell = \Theta_k(n)$, it amounts to showing there are $O_k(n^3)$ terms of $\EE[\nkAP^{\ell,0,S}(y)^6]$ which have a nonzero contribution, i.e., even exponents of the $y$'s.

Each term within this sum is chosen via $d_1,\ldots,d_6\in [n/2]$, each of which contributes the terms $y_{d_jS}$. Now for any valid tuple $(d_1,\ldots,d_6)$, make a graph on vertex set $\{1,\ldots,6\}$, with $i,j$ connected if $d_iS$ and $d_jS$ intersect.

Each term within this sum is chosen via $d_1,\ldots,d_6\in [n/2]$ and offsets (e.g. which term equals the $y_0$ term that is being divided in $\nkAP^{\ell,0,S}$), each of which contributes terms of the form $y_{d_jT_j}$, where $S_j$ is a set which is one of $O_k(1)$ many shifts of $S$ that contain $0$, and $T_j = S_j-\{0\}$. Now for any valid tuple $(d_1,\ldots,d_6)$, make a graph on vertex set $\{1,\ldots,6\}$, with $i,j$ connected if $d_iT_i$ and $d_jT_j$ intersect.

Given such a graph, we claim that there are $O_k(n)$ ways to choose the $d$s associated to a connected component of this graph in a manner compatible with the graph. Indeed, we find that, for example, if $d_1,d_2,d_3$ are connected, then choosing $d_1$ will fix the value say of $cd_2$ for some $c\in [k]$, which yields $O_k(1)$ possible values of $d_2$, and then $O_k(1)$ possible values of $d_3$ similarly.

Furthermore, in a graph associated to a valid tuple $(d_1,\ldots,d_6)$, i.e., one with even powers of the $y$s, we must have no disconnected vertices. Indeed, since the vertex is disconnected, the associated value of $d_i$ must give rise to a multiset $d_iS\pmod{n}$ which must cancel out all of its contributions to the $y$'s. This is impossible as $\gcd(n,(k-1)!) = 1$.

Finally, we have at most $6/2 = 3$ connected components of non-isolated vertices, each of which have $O_k(n)$ ways to choose the $d$'s. This yields an upper bound of $O_k(n^3)$, as desired.

\subsection*{Conclusion}
Putting the various estimates together we have proved the following result, which is a restatement of \cref{thm:scaling-central-limit}.
\begin{theorem}\label{thm:scaling-central-limit-2}
For $g\in C^3(\RR^{k-1})$, we have
\[|\EE g(W) - \EE g(Z)|\lesssim_{k}\frac{M_2(g)+M_3(g)}{n^{1/2}},\]
where $Z$ is a standard Gaussian random vector in $\RR^{k-1}$.
\end{theorem}

\section{Conversion to Bound in Kolmogorov Distance}\label{sec:kolmogorov-distance}
We now convert this test function bound into a bound on the cumulative distribution function.
\begin{lemma}\label{lem:gaussian-pair}
For any $a, b$ we have that 
\[\mathbb{P}\bigg[\nkAP^{1}(y)<a,\frac{\sum_{i=3}^{k}\sigma_ip^{k-\frac{i}{2}}(1-p)^{\frac{i}{2}}\nkAP^{i}(y)}{\sqrt{\sum_{i=3}^{k}\sigma_i^2p^{2k-i}(1-p)^i}}<b\bigg] = \frac{1}{2\pi}\int_{-\infty}^{a}\int_{-\infty}^{b}e^{\frac{-(x^2+y^2)}{2}} 
~\text{dx}~\text{dy} + O_{k,p}(n^{-1/8}).\]
\end{lemma}
\begin{proof}
The key idea is to take $\phi_{\ell}$ which is a smooth function which is $1$ on $(-\infty, \ell]$, $0$ on $[\ell + \epsilon,\infty)$, has second derivative bounded by $O(\epsilon^{-2})$, and third derivative bounded by $O(\epsilon^{-3})$. Furthermore let $\phi_{\ell}$ be in $[0,1]$ over the entire domain. Given this define $\gamma_{a,b}(x,y) = \phi_a(x)\phi_b(y)$. It follows from $\sigma_i = \Theta_k(n)$ for all $3\le i\le k$ that
\[\Psi_{a,b}(y_1,y_3,\ldots,y_k) = 
\gamma_{a,b}\bigg(y_1,\frac{\sum_{i=3}^{k}\sigma_i p^{k-\frac{i}{2}}(1-p)^{\frac{i}{2}}y_i}{\sqrt{\sum_{i=3}^{k} \sigma_i^2p^{2k-i}(1-p)^i}}\bigg) \]
has $M_2(\Psi_{a,b}) = \Theta_{k,p}(\epsilon^{-2})$, $M_3(\Psi_{a,b}) = \Theta_{k,p}(\epsilon^{-3})$, and 
$\|\partial^\beta \Psi_{a,b}\| = \Theta_{k,p}(\epsilon^{-3})$ for all $|\beta| = 3$. The key point of course is that for standard Gaussians $z_i$ we have 
\[\mathbb{E}[\Psi_{a,b}(z_1,z_3,\ldots,z_k)] = \frac{1}{2\pi}\int_{-\infty}^{a}\int_{-\infty}^{b}e^{\frac{-(x^2+y^2)}{2}} 
~\textit{dx}~\textit{dy} + O_k(\epsilon)\]
and by \cref{thm:invariance,thm:scaling-central-limit} it follows that
\[\mathbb{E}[\Psi_{a,b}( \nkAP^{1}(y),\nkAP^{3}(y),\ldots,\nkAP^{k}(y))] = \mathbb{E}[\Psi_{a,b}(z_1,z_3,\ldots,z_k)] + O_{k,p}\bigg(\frac{\epsilon^{-2}+\epsilon^{-3}}{n^{1/2}}\bigg).\] Choosing $\epsilon = n^{-1/8}$, we obtain the desired right side as an upper bound, noting that $\Psi_{a,b}$ dominates the desired indicator function. A lower bound is obtained in an analogous manner and the result then follows.
\end{proof}

\section{Proof of the Failure of Local Central Limit Theorem}\label{sec:local-clt-failure}
In this section we prove that $\kAP$ does not obey a local central limit theorem.  Specifically:
\begin{reptheorem}{thm:local-clt-failure}
Fix $p\in (0,1)$ and $k\ge 3$. Then for all sufficiently large $n$ relatively prime to $(k-1)!$ there is a point $x$ such that \[\Bigg|\mathbb{P}[\kAP(z) = x] - \frac{1}{\sigma_n\sqrt{2\pi}}\exp\bigg(\frac{-(x-\mu_n)^2}{2\sigma_n^2}\bigg)\Bigg| = \Omega\bigg(\frac{1}{\sigma_n}\bigg),\]
where $\mu_n$ and $\sigma_n$ in the statement are the expectation and standard deviation of $\kAP(z)$ and $z$ is sampled i.i.d. with probability $p$.
\end{reptheorem}
We first give a high level overview of the proof.  The proof proceeds by building two sets  $L_\alpha$ and $L_\beta$ of equal size both close enough to the mean of $\kAP$ so that were there to be a local limit theorem for $\kAP$ we would necessarily have
\[\PP[\kAP \in L_\alpha]\approx\PP[\kAP \in L_\beta]\approx\frac{|L_\alpha|}{\sigma_n\sqrt{2\pi}}.\]
However, as a result of Lemma \ref{lem:gaussian-pair} we will be able to compute
$\PP[\kAP\in L_\alpha]\propto f_\delta(\alpha)$, and $\PP[\kAP \in L_\beta]\propto f_\delta(\beta)$, where $f_\delta$ is as defined in Lemma \ref{lemma:fourier} below.

The rough idea in building $L_\alpha$ and $L_\beta$ is to note that $\kAP^1+\kAP^2$ takes values very near to a lattice $G\mathbb{Z}$ where $G=\Theta(n)$.  Meanwhile $\kAP^{>2}$ has standard deviation $\Theta(n)$.  $L_\alpha$ (and $L_\beta$) will roughly correspond to the event that $\kAP\approx \alpha G \mod G$ ($\approx \beta G \mod G$ respectively), and we can use our joint central limit theorem for $\kAP^1$ and $\kAP^{>2}$ to show that $\PP[\kAP \approx \alpha G \mod G]\propto f_\delta(\alpha)$ for some choice of $\delta$.  But since $f_\delta(\alpha)\neq f_\delta(\beta)$ we will conclude that $\PP(L_\alpha)\neq \PP(L_\beta)$ disproving any chance that $\kAP$ obeys a local central limit theorem.

Before beginning our proof, we define $f_\delta$ and prove it is nonconstant.

\begin{lemma}\label{lemma:fourier}
The function 
\[f_\delta(x) = \sum_{\lambda\in\mathbb{Z}} e^{\frac{-(x-\lambda)^2}{\delta}}\]
is not the constant function for any value $\delta > 0$.  Furthermore for any $C>0$ there is a constant $D>0$ so that if $\delta<C$ then for some pair $\alpha,\beta\in (0,1)$ we have $f_\delta(\alpha)-f_\delta(\beta)\ge D$.
\end{lemma}

\begin{proof}
Let us calculate the Fourier transform of $f=f_\delta$ as $1$-periodic function. Note that 
\begin{align*}
\hat{f}(n) &= \int_{[0,1]} e^{-2\pi inx} f(x)~\text{dx} = \int_{[0,1]} e^{-2\pi inx} \sum_{\lambda\in\mathbb{Z}} e^{\frac{-(x-\lambda)^2}{\delta}}~\text{dx}\\
&= \int_{\mathbb{R}} e^{-2\pi inx}  e^{\frac{-x^2}{\delta}}~\text{dx} = \sqrt{\pi \delta}e^{-\pi^2n^2\delta},
\end{align*}
and therefore the function is not constant. Furthermore, by Parseval we have that
\begin{align*}
\int_{[0,1]} (f(x)-\hat{f}(0))^2 dx&=\sum_{n=1}^\infty \hat{f}^2(n)=\sum_{n=1}^\infty \pi \delta e^{-2\pi n^2\delta}\ge \pi \int_\delta^\infty e^{-2\pi\frac{x^2}{\delta}}dx
\end{align*}
and so the variance of $f$ is bounded below whenever $\delta$ is bounded, proving the result.
\end{proof}
We now prove our main result, the failure of a local central limit theorem for constant $p$ and $k$.
\begin{reptheorem}{thm:local-clt-failure}
Fix $p\in (0,1)$ and $k\ge 3$. Then for all sufficiently large $n$ relatively prime to $(k-1)!$ there is a point $x$ such that \[\Bigg|\mathbb{P}[\kAP(z) = x] - \frac{1}{\sigma_n\sqrt{2\pi}}\exp\bigg(\frac{-(x-\mu_n)^2}{2\sigma_n^2}\bigg)\Bigg| = \Omega\bigg(\frac{1}{\sigma_n}\bigg),\]
where $\mu_n$ and $\sigma_n$ in the statement are the expectation and standard deviation of $\kAP(z)$ and $z$ is sampled i.i.d. with probability $p$.
\end{reptheorem}

Throughout this section we use $x_i$ to denote the 0,1 indicator of whether $i$ is in our random set, and $y_i$ to denote the normalized Bernoulli random variables $y_i:=(x_i-p)/\sqrt{p(1-p)}$.  We additionally use the shorthand $\ell=\sum_{i=1}^n y_i$ and $\tilde \ell:=\sum_{i=1}^n x_i$ and $q:=1-p$.

First we need exact formulae for $\kAP^1$ and $\kAP^2$ in terms of $\ell$, which are

\begin{align*}
\kAP^{1}&=\sum_{i=1}^n \frac{k(n-1)}{2}p^{k-\frac{1}{2}}q^\frac12 y_i= \frac{k(n-1)}{2}p^{k-\frac{1}{2}}q^{\frac12}\ell\\
\kAP^2&={k\choose 2}p^{k-1}q\sum_{|S|=2}y_S={k\choose 2}\frac{p^{k-1}q}{2}\left(\ell^2-\sum_{i=1}^n y_i^2\right)={k\choose 2}\frac{p^{k-1}q}{2}\left(\ell^2-n-\frac{1-2p}{\sqrt{pq}} \ell\right)
\end{align*}
where $y_S = \prod_{i\in S}y_i$. And so we find that we can express $\kAP^1+\kAP^2=C_0+C_1\ell+C_2\ell^2:=Q(\ell)$ where
\begin{align*}
C_0:&=-n\frac{k(k-1)}{4}p^{k-1}q\\
C_1:&=\frac{k(n-1)}{2}p^{k-\frac12}q^\frac12-\frac{(1-2p)k(k-1)}{4}p^{k-\frac32}q^\frac12\\
C_2:&=\frac{k(k-1)}{4}p^{k-1}q
\end{align*}

So to understand $\kAP^1+\kAP^2$ it is enough to understand $\ell$.  We note that $\ell$ is valued on the lattice $-n\sqrt{p/q}+\mathbb{Z}/\sqrt{pq}$.  The most commonly taken value for $\ell$ occurs when $\tilde \ell=\sum_{i=1}^n x_i=[pn]$.  When this occurs we see that $\ell$ takes the value 
$$a_0:=\frac{\tilde \ell-pn}{\sqrt{pq}}=\frac{[pn]-pn}{\sqrt{pq}}$$
and hence $\kAP^1+\kAP^2$ takes the value
$$x_0:=\kAP^1+\kAP^2=Q( a_0)=C_0+C_1a_0+C_2{a_0}^2$$
Now we can define $X=\kAP^1(y)+\kAP^2(y)-x_0$ and $Y=\kAP^{\ge 3}(y)$.  So $X+Y=\kAP-\mu-x_0$, and the most likely value taken by $X$ is 0.

Let  $\ldots,A_{-2},~A_{-1},~A_0=0,~A_1,\ldots$ be the values taken by $X$ when $|\ell|\lesssim_{k,p} n$, listed in order. In general we see that $X$ takes the value $A_t$ whenever $\tilde \ell=[pn]+t$ and so $\ell =a_0+t/\sqrt{pq}$ when $|\ell|\lesssim_{k,p} n$ (as $Q'(\ell) > 0$ in such a range). Therefore we can compute for any $t$ that
\begin{align*}
A_t&=Q\left(a_0+\frac{t}{\sqrt{pq}}\right)-Q(a_0)=C_2\left(\frac{2ta_0}{\sqrt{pq}}+\frac{t^2}{pq}\right)+C_1\frac{t}{\sqrt{pq}}
\end{align*}
To alleviate notation we define a constant $G$ for the dominant increment  $G:=C_1/\sqrt{pq}$. 
It will commonly be helpful to have the bound
\[|A_t-tG|=\left|C_2\left(\frac{2ta_0}{\sqrt{pq}}+\frac{t^2}{pq}\right)\right|\le C_2\frac{t^2+2t}{pq}.\]
Now we are in a position to define the events we look at.  Fix $\alpha\in [0,1]$, and  $i,B\in \mathbb{R}$.  Then we define the intervals $I_\alpha(i,B)$ and families of intervals $L_\alpha(B,s)$ by  setting
\begin{align*}
I_\alpha(i,B)&:=\left[G(i+\alpha)-B,~G(i+\alpha)+B\right]\\
L_\alpha(B,s)&:=\bigcup_{i=-s}^s I_{\alpha}(i,B)
\end{align*}

Note that if $\alpha=0$ then the intervals $I_0(i,B)$ are just intervals of length $2B$ centered around that lattice points in $C_1\mathbb{Z}/\sqrt{pq}$.  These intervals are disjoint so long as we ensure $|B|<G/2$.  As $\alpha$ moves between $0$ and $1$, the interval slides between neighboring lattice points in $G\mathbb{Z}$. $L_\alpha(B,s)$ collects the most central $2s+1$ intervals in the collection.

Our goal becomes to show that for some pair $B,s=o(n)$ there exist distinct values $\alpha,\beta\in (0,1)$ so that $\PP[X+Y\in L_\alpha(B,s)]$ and $\PP[X+Y\in L_\beta(B,s)]$ are far. This will contradict the existence of a local limit theorem, thus finishing our proof of \cref{sec:local-clt-failure}.

We first choose $\eta$ so that $\PP[|Y|\ge \eta C_1/\sqrt{pq}]\le n^{-100}$.  By hypercontractivity concentration bounds (see e.g. \cite[Theorem~10.24]{O14}) we see that we may take $\eta = \Theta_{p,k}((\log n)^{k/2})$.

We may use Lemma \ref{lem:gaussian-pair} and some work to compute $\PP[X+Y\in L_\alpha(B,s)]$. We state our result and postpone the calculation to \cref{app:compute}.

\begin{lemma}\label{LAlphaEstimator2}
Assume that $0<B<n^{1-1/36}$ and $\eta<s<n^{1/2-1/24}$.  Let $\sigma_Y$ denote the standard deviation of $Y$ (and as such $\sigma_Y=\Theta(n)$). Then 
\[\PP\left[X+Y\in L_\alpha(B,s)\right] =\frac{2sB\sqrt{2}}{\sigma_Y\sqrt{\pi npq}}f_\delta(\alpha)+O\left(\eta n^{-1/8}\right)\]
for some uniformly bounded $\delta = \delta(n, k, p)$.
\end{lemma}
\begin{remark}
In fact, as $n\to\infty$ this function tends towards a limit; we do not bother replacing $\delta(n, k, p)$ with its limit $\delta(k, p)$ as this will be inconsequential to our arguments.
\end{remark}

Additionally, were $\kAP$ to obey a local limit theorem, then we could compute $\PP[X+Y\in L_\alpha(B,s)]$ in a different way.

\begin{lemma}\label{LLTIntervalEstimator}
Assume that $Z$ is a random variable with mean $\mu_Z$ and standard deviation $\sigma_Z$ which for all $m\in \mathbb{N}$ satisfies 
$$\PP[Z=m]=\frac{1}{\sqrt{2\pi }\sigma_Z}\exp\left(\frac{(m-\mu_Z)^2}{2\sigma_Z^2}\right)+o(\sigma_Z^{-1})$$
Then for any set $S\subset \mathbb{Z}\cap [\mu-T,~\mu+T]$ with $T\le\sigma_Z$ we have 
\[\PP(Z\in S)=\frac{|S|}{\sqrt{2\pi }\sigma_Z}+o\left(\frac{|S|}{\sigma_Z}\right)+O\left(\frac{|S|T}{\sigma_Z^2}\right).\]
\end{lemma}
\begin{proof}
This follows from simply noting that $|1-\exp((m-\mu)^2/2\sigma_Z^2)|\lesssim (m-\mu)/\sigma_Z$ in this range, and summing over $\sum_{m\in S}\PP(Z=m)$.
\end{proof}

We now have all of the tools to prove \cref{thm:local-clt-failure}.

\begin{proof}[Proof of \cref{thm:local-clt-failure}]
Choose $\delta = \delta(k, p)$, which is constant. We know from \cref{lemma:fourier} that there are $\alpha \neq \beta \in (0,1)$ so that $f_\delta(\alpha)-f_\delta(\beta) =\Omega(1)$.  If we set $B=\lfloor n^{1-1/36}\rfloor$ and $s=\lfloor n^{1/2-1/24}\rfloor$ then by \cref{LAlphaEstimator2} we have
$$\PP\left[X+Y\in L_\alpha(B,s)\right] =\frac{2sB\sqrt{2}}{\sigma_Y\sqrt{\pi npq}}f_\delta(\alpha)+O\left(\eta n^{-1/8}\right)$$
and likewise for $L_\beta(B,s)$.  Crucially, note that the main term has size $n^{-5/72}$ and so dominates the error term.  
Applying the same reasoning to $L_\beta(B,s)$ and taking differences yields
\begin{align*}
\PP\left[X+Y\in L_\alpha(B,s)\right]&-\PP\left[X+Y\in L_\beta(B,s)\right]=\frac{2sB\sqrt{2}}{\sigma_Y\sqrt{\pi npq}}\left[f_\delta(\alpha)-f_\delta(\beta)\right]+O\left(\eta n^{-1/8}\right)\\
&=\Omega(n^{-5/72})
\end{align*}

However if we assume that $\kAP$ obeys a local central limit theorem, then so does $X+Y$. Note that $L_\alpha(B,s)$ consists of elements of size at most $Gs+B=O(n^{3/2-1/24})$, and $|L_\alpha(B,s)|=\Theta(n^{3/2-5/72})$. So by \cref{LLTIntervalEstimator} we see that
\begin{align*}
\PP[X+Y\in L_\alpha(B,s)]-\PP[X+Y\in L_\beta(B,s)]=o\left(n^{-5/72}\right)
+O\left(n^{-1/9}\right)
\end{align*}
This is a direct contradiction.
\end{proof}

\section{Conclusion}\label{sec:conclusion}
We conclude by pointing out various open questions regarding the failure of the local central limit theorem. First, note that there is a gap between the theorems we proved and a total explanation of the behavior exhibited in Figure \ref{experiment}. An ideal theorem would prove more precisely that the distribution of $\kAP(S)$ tends to the convolution of two discrete Gaussians seemingly exhibited in the figure. Such a precise theorem would prove that $\kAP(S)$, conditional on the size of $S$, satisfies a local central limit theorem. Failing this precise local limit theorem, at least one could hope to understand how wildly the distribution oscillates in the following precise sense:
\begin{question}
Fix $\epsilon = 1/100$. Let $S$ be a subset of $\ZZ/n\ZZ$ where each element is chosen independently with probability $1/2$ and set $\mu_n = \EE[\AP(S)]$. What is the largest constant $C$ such that for infinitely many $n$ there exist integers $x_n,y_n$ such that $|x-\mu_n|,|y-\mu_n| \le n^{\frac{3}{2}-\epsilon}$ which satisfy \[\frac{\PP[\AP(S)=x_n]}{\PP[\AP(S)=y_n]}\ge C?\]
\end{question}
We argue in the subsection below that $C\approx 4.745$ works, however our method does suggest that in fact that $C$ cannot be taken arbitrarily large. To be more precise, the largest constant $C$ coming from our method is $\frac{\sup(f(x))}{\inf(f(x))}$ where $f(x) = \sum_{\lambda \in \ZZ} e^{-9(x-\lambda)^2}$. This is in fact the optimal constant if one can prove a sufficiently strong local limit theorem in the style of our above results. Furthermore, proving that $C$ cannot be taken to be unbounded, along with the central limit theorem in the paper, would immediately show anti-concentration for the random variable $\AP(S)$ at the correct level.

\subsection{Computing $C$}
Here we compute the value of $C$ coming out of \cref{thm:local-clt-failure} by computing the resulting value of $\lambda$ in the case $k = 3, p = \frac{1}{2}$. It is easy to show that
\[8\cdot\AP(y) = \frac{n(n-1)}{2}+\frac{3}{2}(n-1)\sum_{a\in\ZZ/n\ZZ}y_a+3\sum_{0\le a<b<n}y_ay_b+\sum_{a\in\ZZ/n\ZZ}\sum_{d\in[n/2]}y_ay_{a+d}y_{a+2d},\]
hence $\kAP^1(y)$ has steps of size $3n$ (since $y_a\in\{-1,1\}$) while $\kAP^3(y)$ has mean zero and variance $\frac{n(n-1)}{2}$. Hence the associated normal has standard deviation of size $\frac{n}{\sqrt{2}}$. If we normalize $\kAP^1(y)$ to have steps of size $1$, then $\kAP^3(y)$ will be normalized to have standard deviation $\frac{\sqrt{2}}{6}$ hence density proportional to $e^{-9x^2}$. Therefore we find value $\delta = \frac{1}{9}$, and the claimed ratio as above. It can further be shown that the minimum is attained at $\frac{1}{2}+\ZZ$ and the maximum at $\ZZ$, hence we can prove a ratio of
\[C = \frac{\sum_{x\in\ZZ}e^{-9x^2}}{\sum_{x\in\ZZ}e^{-9(x-\frac{1}{2})^2}}\approx 4.745.\]

\bibliographystyle{alpha}
\bibliography{main}

\newcommand{\etalchar}[1]{$^{#1}$}
\begin{thebibliography}{BAKL{\etalchar{+}}19}

\bibitem[BAKL{\etalchar{+}}19]{BKL19}
Yacine Barhoumi-Andr{\'e}ani, Christoph Koch, Hong Liu, et~al.
\newblock Bivariate fluctuations for the number of arithmetic progressions in
  random sets.
\newblock {\em Electronic Journal of Probability}, 24, 2019.

\bibitem[Bera]{B2}
Ross Berkowitz.
\newblock A local limit theorem for cliques in ${G} (n, p)$.
\newblock arXiv:1811.03527.

\bibitem[Berb]{B1}
Ross Berkowitz.
\newblock A quantitative local limit theorem for triangles in random graphs.
\newblock arXiv:1610.01281.

\bibitem[BGSZ20]{BGSZ20}
Bhaswar~B. Bhattacharya, Shirshendu Ganguly, Xuancheng Shao, and Yufei Zhao.
\newblock Upper {T}ail {L}arge {D}eviations for {A}rithmetic {P}rogressions in
  a {R}andom {S}et.
\newblock {\em Int. Math. Res. Not. IMRN}, (1):167--213, 2020.

\bibitem[CCHT]{CCHT18}
Bryce Cai, Annie Chen, Ben Heller, and Eyob Tsegaye.
\newblock Limit {T}heorems for {D}escents in {P}ermutations and {A}rithmetic
  {P}rogressions in $\mathbb{Z}/p\mathbb{Z}$.
\newblock arXiv:1810.02425.

\bibitem[FKS]{FKS19}
Jacob Fox, Matthew Kwan, and Lisa Sauermann.
\newblock Anticoncentration for subgraph counts in random graphs.
\newblock arXiv:1905.12749.

\bibitem[GK16]{GK16}
Justin Gilmer and Swastik Kopparty.
\newblock A local central limit theorem for triangles in a random graph.
\newblock {\em Random Structures Algorithms}, 48(4):732--750, 2016.

\bibitem[HMS]{HMS19}
Matan Harel, Frank Mousset, and Wojciech Samotij.
\newblock Upper tails via high moments and entropic stability.
\newblock arXiv:1904.08212.

\bibitem[JW16]{JW16}
Svante Janson and Lutz Warnke.
\newblock The lower tail: {P}oisson approximation revisited.
\newblock {\em Random Structures Algorithms}, 48(2):219--246, 2016.

\bibitem[Mec09]{M09}
Elizabeth Meckes.
\newblock On {S}tein's method for multivariate normal approximation.
\newblock In {\em High dimensional probability {V}: the {L}uminy volume},
  volume~5 of {\em Inst. Math. Stat. (IMS) Collect.}, pages 153--178. Inst.
  Math. Statist., Beachwood, OH, 2009.

\bibitem[O'D14]{O14}
Ryan O'Donnell.
\newblock {\em Analysis of {B}oolean functions}.
\newblock Cambridge University Press, New York, 2014.

\bibitem[War17]{W17}
Lutz Warnke.
\newblock Upper tails for arithmetic progressions in random subsets.
\newblock {\em Israel J. Math.}, 221(1):317--365, 2017.

\end{thebibliography}

\appendix
\section{Proof of Lemma \ref{LAlphaEstimator2}}\label{app:compute}

We use all the same terminology and notation as in \cref{sec:local-clt-failure}, including $C_0, C_1, C_2$. First we need a lemma relating the probability that $X+Y\in L_\alpha(B,s)$ to a bound on the the joint distribution of $X$ and $Y$.

\begin{lemma} \label{SumLowerBound}
Assume that $s>\eta$ and $ B-C_2\frac{s^2+s}{pq}\ge 0$.  Then
$$\PP\left[X+Y\in L_\alpha(B,s)\right] \ge \PP\left[|X|\le A_{s-\eta},\quad Y\in L_\alpha\big(\eta, B-C_2\frac{s^2+s}{pq}\big)\right]$$
\end{lemma}
\begin{proof}
We show that, in fact, the event on the left hand side contains the event on the right.  Let $X=A_t$ where $|t|\le s-\eta$ and assume that $Y=G(i+\alpha)+b$ where $|i|\le \eta$ and $b\le  B-C_2\frac{s^2+s}{{pq}}$.  Then we note that
\begin{align*}
\left|X+Y-G(t+i+\alpha)\right|&=\left|\left(tG+C_2\left(\frac{2ta_0}{\sqrt{pq}}+\frac{t^2}{pq}\right)+G(i+\alpha)+b\right)-G(t+i+\alpha)\right|\\
&\le C_2\frac{s(s+1)}{pq}+b\le B
\end{align*}
As $(t+i)\le s$ it follows that if the event on the right hand side occurs, then $X+Y\in L_\alpha(B,s)$.
\end{proof}
\begin{lemma} \label{SumUpperBound}
Assume that $|B|<G/2$ and $C_2\frac{(s+\eta+2)(s+\eta+3)}{pq}<G/2$.  Then
$$\PP\left[X+Y\in L_\alpha(B,s)\right] \le \PP\left[|X|\le A_{s+\eta+1},\quad Y\in L_\alpha(\eta, B+C_2\frac{(s+\eta)(s+\eta+1)}{pq})\right]+O\left(n^{-100}\right)$$
\end{lemma}

\begin{proof}
First we note that
\[\PP(X+Y\in L_\alpha(B,S))\le \PP(X+Y\in L_\alpha(B,S)~\mbox{and }|Y|\le \eta G)+O(n^{-100}).\]
So we will throughout condition on the event that $|Y|\le \eta G$. In that event, we show that the event on the left hand side implies that the event on the right hand side occurs. First, we can simply upper bound $|X|$ by 
\[|X|\le |X+Y|+|Y|\le |(s+\alpha)G+B|+\eta G.\]
But we know that 
\[A_{s+\eta+2}\ge (s+\eta+2)G- C_2\frac{(s+\eta+2)(s+\eta+3)}{pq}> (s+\eta)G+\alpha G+B.\]
And so we have that $X=A_t$ where $|t|\le s+\eta+1$.

For $Y$ we note that $X+Y\in L_\alpha(B,s)$ implies that $X+Y=(r+\alpha)G+b$ where $|r|\le s$ and $|b|\le B$.  By the argument above we know that $X=A_t$ for $t\le s+\eta+1$.  And so we can bound
\begin{align*}
\left|Y-(r-t+\alpha)G\right|&=\left|X+Y-X-(r-t+\alpha)G\right|=\left|(r+\alpha)G+b-A_t-(r-t+\alpha)G\right|\\
&\le B+\left|C_2\frac{(s+\eta+1)(s+\eta+2)}{pq}\right|.
\end{align*}
The last thing we need to do to show that $Y\in L_\alpha(\eta, B+C_2\frac{(s+\eta)(s+\eta+1)}{pq})$ is to show that $|r-t|\le \eta$, but were it otherwise the above equation implies that $|Y|> (\eta+1)G-G$, contradicting our assumption that $|Y|\le \eta G$. Therefore it must follow that
$ Y\in L_\alpha(\eta, B+C_2\frac{(s+\eta)(s+\eta+1)}{pq})$.
\end{proof}

The main upshot here is that for well chosen values of $B$ and $s$, the probability bounds furnished by the above lemmas will be indistinguishable up to a margin of error. The last ingredient we need is an estimate for the probability the events $|X|\le A_s$ and $Y\in L_{\alpha}(B,s)$.
\begin{lemma}\label{midpoint}
For any interval of the form $(T-B,~T+B)$ we have
$$\int_{T-B}^{T+B} e^{-t^2/2}dt=2Be^{-T^2/2}+O(B^3)$$
\end{lemma}
\begin{proof}  This is just the midpoint rule combined with the observation that $|\frac{d^2}{dt^2} e^{-t^2/2}|\le 1$. \end{proof}

\begin{lemma}\label{LAlphaEstimator}
Assume that $B<n^{1-1/36}$, $s<n^{1/2-1/24}$, and $\eta\ge \log(n)$.  Let $f_\delta$ be as defined in \cref{lemma:fourier}.  Then
$$\PP\left[|X|\le A_s,~~Y\in L_\alpha(\eta, B)\right]=\frac{2sB\sqrt{2}}{\sigma_Y\sqrt{\pi npq}}f_\delta(\alpha)+O\left(\eta n^{-1/8}\right)$$,
\end{lemma}
where $\delta = pq\sigma_Y^2/C_1^2$.
\begin{proof}
First, we note that $|X|\le A_s$ is equivalent to saying that $\overline{\kAP}^1\in [a_0/\sqrt n-\frac{s}{\sqrt{npq}},~a_0/\sqrt n+\frac{s}{\sqrt{npq}}]$.
First we note that by Lemma \ref{lem:gaussian-pair} for each interval $I_\alpha(i,B)$ we have
$$\PP(|X|\le A_s,~~Y\in I_\alpha(i,B))=\left(\frac{1}{\sqrt{2\pi}}\int_{a_0/\sqrt n-\frac{s}{\sqrt{npq}}}^{a_0/\sqrt n+\frac{s}{\sqrt{npq}}}e^{-t^2/2}dt\right)\left(\frac{1}{\sqrt{2\pi}}\int_{I_\alpha(i,B)/\sigma_Y}e^{-t^2/2}dt\right)+O(n^{-1/8})$$

By Lemma \ref{midpoint} we can estimate both of these integrals. The first is
\begin{align*}
\frac{1}{\sqrt{2\pi}}\int_{a_0/\sqrt n -\frac{s}{\sqrt{npq}}}^{a_0/\sqrt n+\frac{s}{\sqrt{npq}}}e^{-t^2/2}dt=\left(\frac{2s}{\sqrt{npq}}\right)\frac{1}{\sqrt{2\pi}}e^{-a_0^2/2\sigma_1^2}+O\left(\frac{s^3}{(npq)^{3/2}}\right)=\frac{s\sqrt{2}}{\sqrt{\pi npq}}+O\left(\frac{s^3}{n^{3/2}}+\frac{1}{n^3}\right)
\end{align*}
and the second estimate is
\begin{align*}
\frac{1}{\sqrt{2\pi}}\int_{I_\alpha(i,B)\sigma_Y}e^{-t^2/2}dt&=\frac{1}{\sqrt{2\pi}}\int_{\frac{C_1(i+\alpha)}{\sqrt{pq}\sigma_Y}-B/\sigma_Y}^{\frac{C_1(i+\alpha)}{\sqrt{pq}\sigma_Y}+B/\sigma_Y}e^{-t^2/2}dt=\frac{2B}{\sigma_Y}e^{-\frac{C_1^2(i+\alpha)^2}{pq\sigma_Y^2}}+O(B^3/n^3).
\end{align*}
So, defining $\delta:=pq\sigma_Y^2/C_1^2=\Theta(1)$ and combining all of these estimates, we find that
\begin{align*}
\PP[|X|\le A_s,~~Y\in I_\alpha(i,B)]&=\left(\frac{s\sqrt{2}}{\sqrt{\pi npq}}+O\left(\frac{s^3}{n^{3/2}}+\frac{1}{n^3}\right)\right)\left(\frac{2B}{\sigma_Y}e^{-\frac{(i+\alpha)^2}{\delta}}+O(B^3/n^3)\right)+O(n^{-1/8})\\
&=\frac{2sB\sqrt{2}}{\sigma_Y\sqrt{\pi npq}}e^{-\frac{(i+\alpha)^2}{\delta}}+O\left(s^3/n^{3/2}+\frac{B^3s}{n^{3.5}}+n^{-1/8}\right)\\
&=\frac{2sB\sqrt{2}}{\sigma_Y\sqrt{\pi npq}}e^{-\frac{(i+\alpha)^2}{\delta}}+O\left(n^{-1/8}\right).
\end{align*}
Thus taking a union yields
\begin{align*}
\PP\left[|X|\le A_s,~~Y\in L_\alpha(\eta, B)\right]&=\sum_{i=-\eta}^\eta\frac{2sB\sqrt{2}}{\sigma_Y\sqrt{\pi npq}}e^{-\frac{(i+\alpha)^2}{\delta}}+O\left(\eta n^{-1/8}\right)\\
&=\sum_{i=-\infty}^\infty\frac{2sB\sqrt{2}}{\sigma_Y\sqrt{\pi npq}}e^{-\frac{(i+\alpha)^2}{\delta}}+O\left(\eta n^{-1/8}+e^{-\eta^2/\delta}\right)\\
&=\frac{2sB\sqrt{2}}{\sigma_Y\sqrt{\pi npq}}f_\delta(\alpha)+O\left(\eta n^{-1/8}\right).\qedhere
\end{align*}
\end{proof}
\begin{replemma}{LAlphaEstimator2}
Assume that $|B|<n^{1-1/36}$, $s<n^{1/2-1/24}$ and $\eta<s$.  Let $\sigma_Y$ denote the standard deviation of $Y$ (and as such $\sigma_Y=\Theta(n)$). Then 
\[\PP\left[X+Y\in L_\alpha(B,s)\right] =\frac{2sB\sqrt{2}}{\sigma_Y\sqrt{\pi npq}}f_\delta(\alpha)+O\left(\eta n^{-1/8}\right).\]
\end{replemma}

\begin{proof}
First we use \cref{SumUpperBound,LAlphaEstimator} to upper bound
\begin{align*}\PP\left[X+Y\in L_\alpha(B,s)\right] &\le \PP\left[|X|\le A_{s+\eta+1},\quad Y\in L_\alpha(\eta, B+C_2\frac{(s+\eta)(s+\eta+1)}{pq})\right]+O\left(n^{-100}\right)\\
&=\frac{s\left(B+C_2\frac{(s+\eta)(s+\eta+1)}{pq}\right)2\sqrt{2}}{\sigma_Y\sqrt{\pi npq}}f_\delta(\alpha)+O\left(\eta n^{-1/8}\right)\\
&= \frac{2sB\sqrt{2}}{\sigma_Y\sqrt{\pi npq}}f_\delta(\alpha)+O\left(\eta n^{-1/8}+s^3/n^{1.5}\right).
\end{align*}
Next we use \cref{SumLowerBound} to lower bound
\begin{align*}
\PP\left[X+Y\in L_\alpha(B,s)\right]& \ge \PP\left[|X|\le A_{s-\eta},\quad Y\in L_\alpha(\eta, B-C_2\frac{s^2+s}{pq})\right]\\
&=\frac{(s-\eta)\left(B-C_2\frac{s^2+s}{pq}\right)2\sqrt{2}}{\sigma_Y\sqrt{\pi npq}}f_\delta(\alpha)+O\left(\eta n^{-1/8}\right)\\
&=\frac{2sB\sqrt{2}}{\sigma_Y\sqrt{\pi npq}}f_\delta(\alpha)+O\left(\eta n^{-1/8}+s^3/n^{1.5}+\eta B/n^{1.5} \right).
\end{align*}
By our hypotheses on $B,s,$ and $\eta$ the $\eta n^{-1/8}$ term dominates in both inequalities.
\end{proof}

\end{document}